\documentclass[12pt]{article}
\usepackage{amsmath,amssymb,amsthm,amsfonts,mathtools,enumerate,setspace,multicol,caption,subcaption,authblk,tikz}
\usepackage[margin=1in]{geometry}
\usepackage[bottom]{footmisc}
\newtheorem{theorem}{Theorem}
\newtheorem{corollary}{Corollary}

\newtheorem{lemma}{Lemma}

\newtheoremstyle{case}{}{}{}{}{}{:}{ }{}
\theoremstyle{case}

\newtheoremstyle{subcase}{}{}{}{}{}{:}{ }{}
\theoremstyle{subcase}

\newtheoremstyle{fact}{}{}{}{}{}{:}{ }{}
\theoremstyle{fact}

\begin{document}
\title{An Upper Bound on the Independent Domination Number of $k$-Trees}

\author{Enoch Akade, Andrew Pham \footnote{Corresponding Author: andrew.pham@aamu.edu}}
\affil{Department of Physics, Chemistry, and Mathematics, Alabama Agricultural and Mechanical University, Normal, AL, USA}
\date{}
\maketitle

\begin{abstract}
Given a simple, finite, nonempty graph $G=(V(G),E(G))$, a vertex subset $D\subseteq V(G)$ is said to be a dominating set if every vertex $v\in V(G)-D$ is adjacent to a vertex in $D$. The independent domination number $\gamma_i(G)$ is the minimum cardinality among all independent dominating sets of $G$. Since determining the domination number for general graphs is NP-complete, we focus on the class of $k$-trees. Favaron established a  tight upper bound for $1$-trees, while Campos and Wakabayashi determined a tight upper bound for maximal outerplanar graphs, a subclass of $2$-trees. We generalize these results and establish a tight upper bound for the independent domination number of $k$-trees for all $k\in \mathbb{N}$. 
\end{abstract}

\section{Introduction} 
Throughout this paper, we consider only simple, finite, and nonempty graphs. We begin by introducing some definitions and terminology. Let $G=(V(G),E(G))$ be a simple, finite, nonempty graph. Given a vertex $v\in V(G)$, the \textbf{(open) neighborhood} of $v$, denoted by $N_G(v)$, is the set of all vertices adjacent to $v$; that is, $N_G(v)=\{u:uv\in E(G)\}$. The \textbf{closed neighborhood} of $v$ is defined as $N_G[v]=N_G(v)\cup \{v\}$. The \textbf{degree} of $v\in V(G)$ is given by $d_G(v) = |\{u : uv \in E(G)\}|=|N_G(v)|$. A \textbf{$j$-,$j^+$-, and $j^-$-vertex} is a vertex of degree exactly $j$, at least $j$, and at most $j$, respectively. Given a graph $G$, a graph $H$ is a \textbf{subgraph} of $G$ if the vertex and edge sets of $H$ are subsets of those of $G$; that is, $V(H) \subseteq V(G)$ and $E(H) \subseteq E(G)$. If $H$ is a subgraph of $G$, we write $H \subseteq G$. Given a nonempty subset of vertices $S \subseteq V(G)$, the \textbf{induced subgraph} defined by $S$, denoted by $G[S]$, is the subgraph having vertex set $V(G[S])=S$ and containing all edges from $G$ that connect vertices in $S$; that is, $E(G[S])=\{uv \in E(G): u,v \in S\}$. 

A \textbf{dominating set} of a graph $G$ is a vertex subset $D\subseteq V(G)$ such that every vertex $v\in V(G)-D$ is adjacent to a vertex in $D$. If $D$ is both dominating and independent, we say $D$ is an independent dominating set. The \textbf{independent domination number}, denoted $\gamma_i(G)$, of $G$ is the minimum cardinality among all independent dominating sets of $G$, that is, the size of a minimum independent dominating set. Domination and independent domination have applications across many areas such as network monitoring, resource allocation, and infastruction planning, where the goal is efficient coverage of the network. Domination number of graphs has been an active research topic for several decades. Garey and Johnson \cite{garey} showed that determining the domination number of general graphs, known as the dominating set problem, is NP-complete. As determining the domination number can be very difficult, it is of interest to focus on particular classes of graphs. In particular, we focus on the class of $k$-trees and provide a tight upper bound on the independent domination number of $k$-trees for all $k\in\mathbb{N}$. 

A \textbf{tree} is a graph in which any pair of vertices is connected by exactly one path, that is, a connected graph with no cycles. Many aspects of trees and tree-related graphs have been studied extensively in graph theory. In 1969, Beineke and Pippert \cite{beineke} generalized trees into the concept of $k$-trees. For $k\in \mathbb{N}$, a \textbf{$k$-tree} on $n\geq k+1$ vertices is defined as follows: (1) The smallest $k$-tree is the complete graph $K_{k+1}$. (2) Given a $k$-tree $G$, a new $k$-tree can be constructed by adding a new vertex and joining it to the vertices of a $k$-clique subgraph in $G$. Following this definition, we can see that the class of $1$-trees is exactly all trees. Although Beineke and Pippert defined $k$-trees for all $k$, the concept of a $2$-tree was introduced one year earlier in 1968 by Harary and Palmer \cite{harary}. 

\begin{figure}[h]
\begin{center}
\includegraphics[scale=0.5]{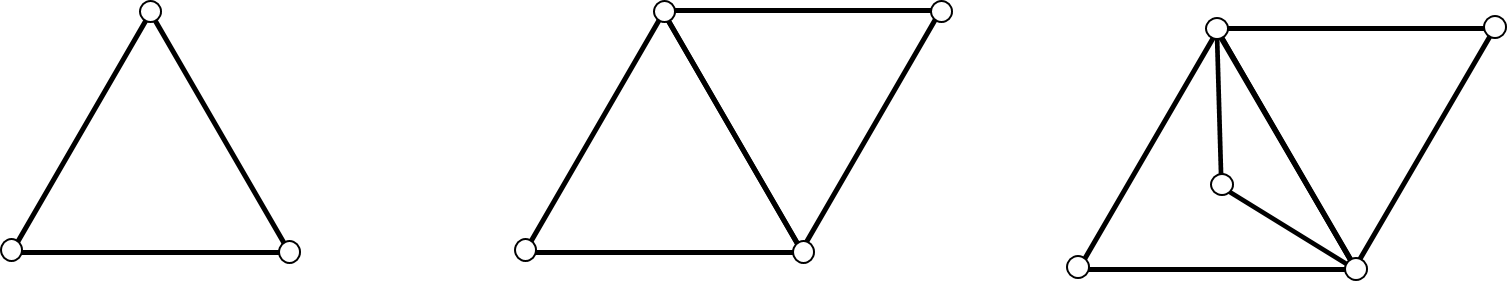}
\caption{Constructing 2-trees from existing 2-trees}
\end{center}   
\end{figure}

The concept of $k$-trees has many applications in Computer Science, Chemistry, Biology, and engineering. In particular, $2$-trees are simple, maximal series-parallel graphs, and can play significant roles in communication and electrical network problems.  Planar $3$-trees, also known as Apollonian networks, are used in models of force chains in granular packings, hierarchical road systems, electrical supply networks \cite{andrade}, and neuronal systems \cite{pellegrini}. We note that some papers define a $k$-tree as a tree with maximum degree $k$; however, throughout this paper, we refer to Pippert's and Beineke's definition of $k$-trees. Further studies on $k$-trees can be found at \cite{duffin} and \cite{shook}.

\section{Preliminary Results}

Favaron \cite{favaron} determined an upper bound on the independent domination number of $1$-trees in respect to the graph's order and its number of vertices of degree one. 
\begin{theorem}\cite{favaron}
If $T$ is a tree with $n$ vertices and $\ell$ vertices of degree $1$, then $$\gamma_i(T)\leq \frac{n+\ell}{3}.$$
\end{theorem}

In general, this bound is tight for trees. Consider a caterpillar graph $G$ constructed by adding a pendant vertex to each vertex on the path $P_m$. Then $G$ has $2m$ vertices and $m$ vertices of degree one. By Theorem 1, $\gamma_i(G)\leq \frac{2m+m}{3}=m$. However, for each pendant vertex $v_i$ for $1\leq i \leq m$, every independent dominating set of $G$ must contain either $v_i$ or its neighbor. As the closed neighborhoods $N[v_i]\cap N[v_j]=\emptyset$ for $i\neq j$, it follows that $\gamma_i(G)\geq m$. Thus, $\gamma_i(G)=m$. 

\begin{figure}[h]
\begin{center}
\includegraphics[scale=0.6]{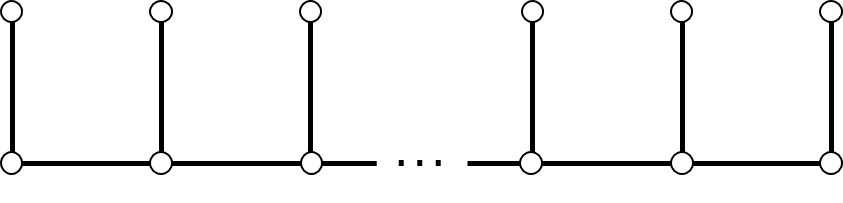}
\caption{The caterpillar graph on $2m$ vertices.}
\end{center}
\end{figure}

In 2013, Tokunaga \cite{tokunaga} and Campos and Wakabayashi \cite{campos} independently proved an upper bound on the domination of maximal outerplanar graphs.

\begin{theorem} \cite{tokunaga,campos}
Suppose $G$ is an $n$-vertex maximal outerplanar graph with $n\geq 3$ having $k$ vertices of degree $2$, then $$\gamma(G)\leq \left\lfloor \frac{n+k}{4}\right\rfloor.$$
\end{theorem}

Campos and Wakabayashi also provided families of maximal outerplanar graphs to demonstrate that their upper bound was tight in general. While Campos and Wakabayashi's proof method considered the graph's internal triangles, Tokunaga's proof instead relied on a simple coloring method. As maximal outerplanar graphs form a subclass of $2$-trees, a natural question is whether this bound and technique can be extended to $2$-trees, and more broadly, to all $k$-trees. We generalize Tokunaga's proof method for independent domination of $k$-trees for all $k\in \mathbb{N}$.

\section{Main Results}

To extend Tokunaga’s coloring-based approach beyond maximal outerplanar graphs, we first show that every $k$-tree has a suitable $(k+2)$-coloring. 

\begin{lemma}
Given a $k$-tree $G$ with at least $k+2$ vertices, there exists a proper $(k+2)$-coloring such that for all $(k+1)^+$-vertices $v\in V(G)$, $N_G[v]$ contains all $k+2$ colors. 
\end{lemma}
\begin{proof}
We proceed by induction on $n=|V(G)|$. For $n=k+2$, $G$ consists of a $k$-vertex attached to a $(k+1)$-clique. In particular, $G$ contains two $k$-vertices and $k$ vertices of degree exactly $k+1$. As each $(k+1)$-vertex is adjacent to all vertices of $G$, its closed neighborhood must contain all $k+2$ colors. The lemma holds true when $n=k+2$. Hence, assume the statement is true for all $k$-trees on $k+1<n$ vertices. Let $G$ be a $k$-tree on $n$ vertices and fix a $k$-vertex $x\in V(G)$. It follows that $G-x$ is a $k$-tree as well. By induction hypothesis, there exists a proper $(k+2)$-coloring of $G-x$ such that the closed neighborhood of each $(k+1)^+$-vertex contains all $k+2$ colors. We will construct a such proper $(k+2)$-coloring of $G$ by using the proper $(k+2)$-coloring of $G-x$ as a base. Let $N_{G}(x)=\{u_1, u_2,\ldots, u_k\}$. We consider two cases:\\\\
\textit{Case 1: } $d_{G-x}(u_i)=k$ for some $1\leq i \leq k$. 

Then $d_G(u_i)=k+1$. In particular, there exists only one such $u_i$. It follows that $G[N_{G-x}[u_i]]\cong K_{k+1}$. Thus, $N_{G-x}[u_i]$ must contain exactly $(k+1)$ colors. Hence, there exists one remaining color that we may assign $x$ in order to form a proper $(k+2)$-coloring of $G$ such that that $N_G[u_i]$ contains all $k+2$ colors.\\

\textit{Case 2: } $d_{G-x}(u_i)>k$ for all $1\leq i \leq k$. 

Then for each $1\leq i \leq k$, $N_{G-x}[u_i]$ contains all $k+2$ colors. As $G$ is a $k$-tree, then $G[N_G(x)]\cong K_{k}$ and $N_G(x)\subsetneq N_{G-x}[u_i]$. Hence, $N_G(x)$ contains exactly $k$ colors, leaving two color options for $x$. We may assign $x$ one of the two remaining colors to form a proper $(k+2)$-coloring of $G$. 
\end{proof}

As the closed neighborhood of every $(k+1)^+$-vertex in $G$ contains all distinct $k+2$ colors, we obtain the following corollary.

\begin{corollary}
Suppose a $k$-tree $G$ has a proper $(k+2)$-coloring $C$ such that for all $(k+1)^+$-vertices $v\in V(G)$, $N_G[v]$ contains all $k+2$ colors. Let $Z_p\subseteq V(G)$ contain all vertices of the color class $p$ in $C$ for $1\leq p \leq k+2$. Then $Z_p$ is independent and dominates all vertices of $G$, except possibly $k$-vertices. 
\end{corollary}

Using Lemma 1 and Corollary 1, we proceed with stating and proving our main result.  

\begin{theorem} \label{main}
Given a $k$-tree $G$ on $n$ vertices, $$\gamma_i(G)\leq \left\lfloor \frac{n+|V^G_k|}{k+2} \right\rfloor.$$ Furthermore, this bound is tight. 
\end{theorem}
\begin{proof}
As $\gamma_i(G)$ only takes on natural number values, it suffices to prove $\gamma_i(G)\leq \frac{n+|V^G_k|}{k+2}.$ If $n=k+1$, then $G\cong K_{k+1}$ and $\gamma_i(G)=1\leq \frac{(k+1)+(k+1)}{k+2}=\frac{2k+2}{k+2}=\frac{k}{k+2}+1$. Hence, assume $n>k+1$. By Lemma 1, there exists a proper $(k+2)$-coloring $C$ of $G$ such that the closed neighborhood of each $(k+1)^+$-vertex contains all $k+2$ colors. For $1\leq p \leq k+2$, let $Z_p \subseteq V(G)$ contain all the vertices with color $p$. For each $p$, define $S_p=V_k^G - N[Z_p]$; that is, $S_p$ is the set of vertices in $V_k^G$ not dominated by $Z_p$. 

For each $v\in V_k^G$, $G[N[v]]\cong K_{k+1}$ by the construction restrictions of a $k$-tree. Hence, $N[v]$ must contain exactly $k+1$ colors. Therefore, each $v$ is not dominated by exactly one color of $C$. It follows that each $v$ belongs to only one $S_p$. Hence, $S_p\cap S_q=\emptyset$ for $p\neq q$. Furthermore, as $n>k+1$, no pair of distinct $k$-vertices are adjacent. By Corollary 1 and our construction of $S_p$, $Z_p\cup S_p$ is an independent dominating set of $G$ for each $1\leq p \leq k+2$. Thus for each $p$, $\gamma_i(G)\leq |Z_p\cup S_p|$.

As $S_p\cap S_q=\emptyset$ for $p\neq q$, we have $$\sum_{p=1}^{k+2} |Z_i\cup S_i|=\sum_{i=1}^{k+2} |Z_i| + \sum_{i=1}^{k+2} |S_i|=n+|V_k^G|.$$

We may consider the average size of our independent dominating sets across the $k+2$ color classes. Therefore,  
$$\gamma_i(G)\leq \frac{1}{k+2}\sum_{p=1}^{k+2} |Z_i\cup S_i|=\frac{n+|V_k^G|}{k+2}.$$
\end{proof}

A $k$-path is a $k$-tree with exactly two $k$-vertices. To prove the tightness of Theorem \ref{main}, consider a $k$-path $P$ with vertex set $V(P)=\{v_{i,j}|1\leq i \leq t, 1\leq j \leq k\}$ for some $t\in \mathbb{N}$ such that $G[\{v_{i,1},v_{i,2},\ldots,v_{i,k-1},v_{i,k}\}]\cong K_k$ for each $1\leq i \leq t$. Then $P$ has $kt$ vertices and contains exactly $t$ vertex disjoint $k$-clique subgraphs. Construct a $k$-tree $P'$ by attaching a $k$-vertex $x_i$ for each $p$ to the disjoint $k$-clique subgraph formed by $\{v_{i,1},v_{i,2},\ldots,v_{i,k-1},v_{i,k}\}$ of $P$. (See Figure 3 for an example where $k=3$ and $t=4$.) Since $N_{P'}[x_i]\cap N_{P'}[x_j]=\emptyset$ for distinct $i\neq j$, then any minimum dominating set of $I$ of $P'$ must contain at least one vertex from $N[x_i]$ for each $1\leq i \leq t$. Hence $\gamma_i(P')\geq \gamma(P')\geq t$. By Theorem 3, $\gamma_i(P')\leq \frac{n+|V_k^{P'}|}{k+2}=\frac{(kt+t)+t}{k+2}=t$. Hence, $\gamma_i(P')=t$.

\begin{figure}[h]
\begin{center}
\includegraphics[scale=0.6]{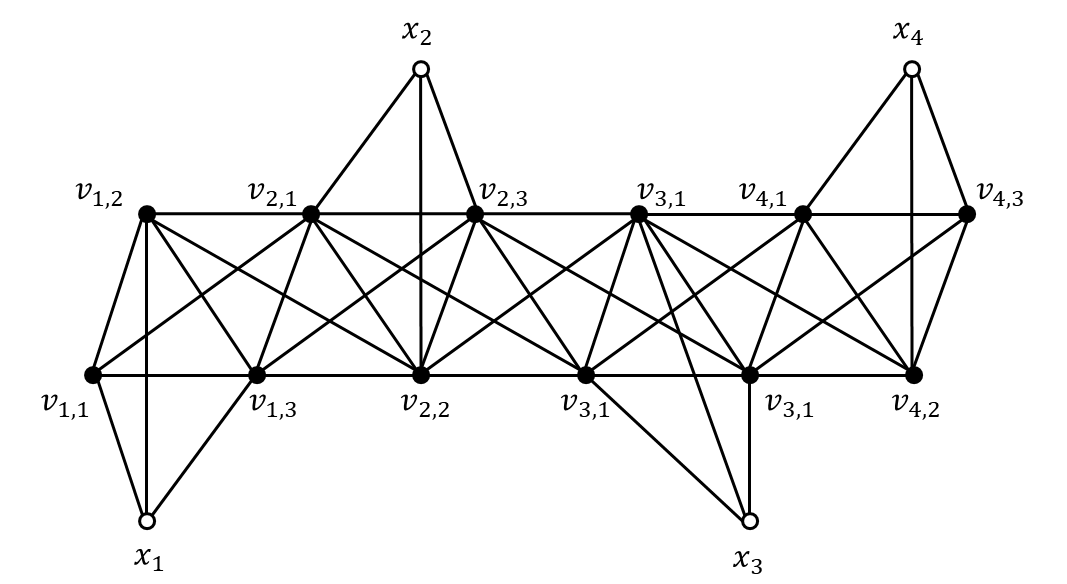}
\caption{Constructing a 3-tree from the 3-path with $t=4$.}
\end{center}
\end{figure}

\section{Concluding Remarks and Future Work}
Throughout this paper, we refer to the recursive definition of the $k$-tree; however, there exists an equivalent characterization based on treewidth. The treewidth of a graph is the minimum width among all possible tree decompositions of $G$, where the width of a decomposition is one less than the size of its largest bag. A graph is called a partial $k$-tree if it has treewidth at most $k$. A $k$-tree is exactly an edge maximal partial $k$-tree. One natural route of future work is to extend our results to the broader class of partial $k$-trees. Examples of partial $k$-trees problems and applications have been studied by Granot and Skorin-Kapove \cite{granot}, Telle and Proskurowski \cite{telle}, and Araki et al. \cite{araki}.

Another natural route is to determine the independent bondage number of $k$-trees. Given a nonempty graph $G$, an edge subset $B\subset E(G)$ is a independent bondage set if the deletion of the edges in $B$ results in a strictly larger independent domination number; that is $\gamma_i(G-B)>\gamma_i(G)$. The independent bondage number $b_i(G)$ is the minimum size of an independent bondage set. The independent bondage number has only been studied in depth in recent years. In 2018, Priddy, Wang, and Wei determined the exact value for some classes of graphs. Pham and Wei \cite{pham}, as well as Gamlath, Reid, and Wei \cite{gamlath1}, established constant upper bounds for planar graphs with minimum degree at least three. In a separate paper, Gamlath, Pham, and Wei \cite{gamlath2} obtained upper bounds for planar graphs under girth constraints. Additionally, Kerdjoudja and Chellalia \cite{samia} determined upper bounds under certain cycle conditions. As most $k$-trees are not planar, it is of interest to see if bounds can be determined for $k$-trees.

\end{document}